\documentclass[oneside,english]{amsart}
\usepackage[T1]{fontenc}
\usepackage[latin9]{inputenc}
\usepackage{mathrsfs}
\usepackage{amstext}
\usepackage{amsthm}
\usepackage{amssymb}
\usepackage[all]{xy}

\makeatletter
\numberwithin{equation}{section}
\numberwithin{figure}{section}
\theoremstyle{plain}
\newtheorem{thm}{\protect\theoremname}
\theoremstyle{definition}
\newtheorem{defn}[thm]{\protect\definitionname}
\theoremstyle{plain}
\newtheorem{prop}[thm]{\protect\propositionname}
\theoremstyle{remark}
\newtheorem{rem}[thm]{\protect\remarkname}
\theoremstyle{plain}
\newtheorem{lem}[thm]{\protect\lemmaname}
\theoremstyle{definition}
\newtheorem{example}[thm]{\protect\examplename}
\theoremstyle{plain}
\newtheorem{cor}[thm]{\protect\corollaryname}

\usepackage[dvipsnames,svgnames,x11names,hyperref]{xcolor}
\usepackage{lmodern}
\renewcommand*{\epsilon}{\varepsilon}
\usepackage{amssymb}
\usepackage{amsfonts}
\usepackage{egothic}
\usepackage{xcolor}
\usepackage[T1]{fontenc}

\usepackage{url}
\usepackage{amsthm}
\usepackage{aliascnt}
\usepackage{lipsum}
\usepackage{mathrsfs}
\usepackage{esint}
\usepackage{url}
\usepackage{amsmath}
\usepackage{dsfont}
\usepackage{bbm}
\usepackage{pdflscape}
\usepackage{bbm}
\usepackage{bussproofs}

\makeatletter
\def\matrixobject@{%
  \edef \next@{={\DirectionfromtheDirection@ }}%
  \expandafter \toks@ \next@ \plainxy@
  \let\xy@@ix@=\xyq@@toksix@
  \xyFN@ \OBJECT@}
\let\xy@entry@@norm=\entry@@norm
\def\entry@@norm@patched{%
  \let\object@=\matrixobject@
  \xy@entry@@norm }
\AtBeginDocument{\let\entry@@norm\entry@@norm@patched}
\makeatother

\newcommand{\twocong}[2][0.5]{\ar@{}[#2] \save ?(#1)*{\cong}\restore}
\newcommand{\twoeq}[2][0.5]{\ar@{}[#2] \save ?(#1)*{=}\restore}
\newcommand{\ltwocell}[3][0.5]{\ar@{}[#2] \ar@{=>}?(#1)+/r 0.15cm/;?(#1)+/l 0.15cm/^{#3}}
\newcommand{\rtwocell}[3][0.5]{\ar@{}[#2] \ar@{=>}?(#1)+/l 0.15cm/;?(#1)+/r 0.15cm/^{#3}}
\newcommand{\utwocell}[3][0.5]{\ar@{}[#2] \ar@{=>}?(#1)+/d  0.15cm/;?(#1)+/u 0.15cm/_{#3}}
\newcommand{\dtwocell}[3][0.5]{ \ar@{}[#2] { \ar@{=>}?(#1)+/u  0.15cm/;?(#1)+/d 0.15cm/^{#3}}}
\newcommand{\ultwocell}[3][0.5]{\ar@{}[#2] \ar@{=>}?(#1)+/dr  0.15cm/;?(#1)+/ul 0.15cm/^{#3}}
\newcommand{\urtwocell}[3][0.5]{\ar@{}[#2] \ar@{=>}?(#1)+/dl  0.15cm/;?(#1)+/ur 0.15cm/^{#3}}
\newcommand{\dltwocell}[3][0.5]{\ar@{}[#2] \ar@{=>}?(#1)+/ur  0.15cm/;?(#1)+/dl 0.15cm/^{#3}}
\newcommand{\drtwocell}[3][0.5]{\ar@{}[#2] \ar@{=>}?(#1)+/ul  0.15cm/;?(#1)+/dr 0.15cm/^{#3}}

\newcommand{\myar}[2]{\ar^-{#1}[#2]}
\newcommand{\myard}[2]{\ar_-{#1}[#2]}

\usepackage[colorlinks=true,linkcolor=Dark Red, citecolor=Dark Green,linktoc=all]{hyperref}
\usepackage{amsmath}
\usepackage{cleveref}

\makeatletter

  \def\make@df@tag@@#1{%
    \gdef\df@tag{%
      \maketag@@@{\Hy@make@anchor#1}%
      \def\@currentlabel{#1}%
      \def\cref@currentlabel{[equation][2147483647][]#1}%
    }%
  }
  \def\make@df@tag@@@#1{%
    \gdef\df@tag{%
      \tagform@{\Hy@make@anchor#1}%
      \toks@\@xp{\p@equation{#1}}%
      \edef\@currentlabel{\the\toks@}%
      \edef\cref@currentlabel{[equation][2147483647][]\the\toks@}
    }%
  }
\makeatother

\numberwithin{thm}{subsection}

\makeatother

\usepackage{babel}
\providecommand{\corollaryname}{Corollary}
\providecommand{\definitionname}{Definition}
\providecommand{\examplename}{Example}
\providecommand{\lemmaname}{Lemma}
\providecommand{\propositionname}{Proposition}
\providecommand{\remarkname}{Remark}
\providecommand{\theoremname}{Theorem}

\begin{document}
\title{Terms and Derivatives of Polynomial functors via Negation}
\author{Charles Walker}
\address{Department of Mathematics and Statistics, Masaryk University, Kotl{\'a}{\v r}sk{\'a}
2, Brno 61137, Czech Republic }
\email{\tt{charles.walker.math@gmail.com}}
\begin{abstract}
Given a locally cartesian closed category $\mathcal{E}$, a polynomial
$\left(s,p,t\right)$ may be defined as a diagram consisting of three
arrows in $\mathcal{E}$ of a certain shape. In this paper we define
the homogeneous and monomial terms comprising a polynomial $\left(s,p,t\right)$
and give sufficient conditions on $\mathcal{E}$ such that the homogeneous
terms of polynomials exist. We use these homogeneous terms to exhibit
an infinite family of coreflection adjunctions between polynomials
and homogeneous polynomials of order $n$.

We show that every locally cartesian closed category $\mathcal{E}$
with a strict initial object admits a negation operator and a $\left(\textnormal{dense},\textnormal{closed}\right)$
orthogonal factorization system. We see that both terms and derivatives
of polynomial functors are constructed from this negation operator,
and that if one takes the localization of $\mathcal{E}$ by the class
$\mathcal{W}$ of dense monomorphisms, then derivatives of all polynomial
functors exist.

All results are shown formally using only the theory of extensive
categories and distributivity pullbacks. 
\end{abstract}

\maketitle
\tableofcontents{}

\section{Introduction}

In a locally cartesian closed category $\mathcal{E}$, a polynomial
is defined as a diagram of the form 
\[
\xymatrix@R=0.5em{ & E\ar[r]^{p}\ar[ld]_{s} & B\ar[rd]^{t}\\
I &  &  & J
}
\]
which yields an associated polynomial functor given by

\[
\xymatrix{\mathcal{E}/I\ar[r]^{\Delta_{s}} & \mathcal{E}/E\ar[r]^{\Pi_{p}} & \mathcal{E}/B\ar[r]^{\Sigma_{t}} & \mathcal{E}/J}
\]
defined by the assignation
\[
\left(X_{i}\colon i\in I\right)\mapsto\left(\sum_{b\in B_{j}}\Pi_{e\in E_{B}}X_{s\left(e\right)}:j\in J\right)
\]
When $\mathcal{E}=\mathbf{Set}$ one may think of $I$ as the set
of input variables, and $J$ as the number of outputs. In the simpler
case in which $I=J=\mathbf{1}$ is the terminal object, a polynomial
is simply a morphism $p\colon E\to B$, and the corresponding polynomial
functor is
\[
\mathcal{E\to\mathcal{E}\colon}X\mapsto\sum_{b\in B}X^{E_{b}}
\]
which is a familiar polynomial in one variable, involving one input
and one output. In the category of sets, it is simple to read off
the terms of the polynomial. For example the coefficient of the $X^{2}$
term is the number of $p$-fibres of size $2$. More generally, we
wish to know how to read off the terms of a polynomial in a general
locally cartesian closed category $\mathcal{E}$, for which we must
first know the definition of a term, and if terms of polynomials even
exist in this general case.

To do this, we first observe that certain constructions on polynomial
functors are defined as terminal diagrams of a certain form. An example
of this, due to Weber \cite{weber}, is that the composite of two
polynomial functors is defined as the terminal diagram of the form
below
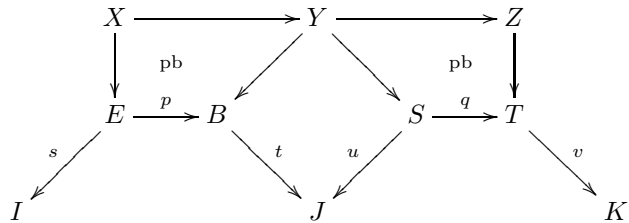
\[
\xymatrix{ & X\ar[d]\ar[rr] & \ar@{}[ld]|-{\textnormal{pb}} & Y\ar[rd]\ar[rr]\ar[ld] & \ar@{}[dr]|-{\textnormal{pb}} & Z\ar[d]\\
 & E\ar[r]^{p}\ar[ld]_{s} & B\ar[rd]^{t} &  & S\ar[r]^{q}\ar[ld]_{u} & T\ar[rd]^{v}\\
I &  &  & J &  &  & K
}
\]
meaning any other such diagram is equal to the above with a unique
pullback placed on top. In a similar fashion, we will define the order
$F$ homogeneous term of a polynomial $\left(s,p,t\right)$ to be
the terminal diagram of the form
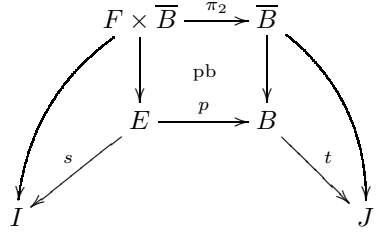
\[
\xymatrix{ & F\times\overline{B}\myar{\pi_{2}}{r}\ar[d]\ar@/_{1pc}/[dld]\ar@{}[dr]|-{\textnormal{pb}} & \overline{B}\ar[d]\ar@/^{1pc}/[ddr]\\
 & E\ar[r]^{p}\ar[ld]_{s} & B\ar[rd]^{t}\\
I &  &  & J
}
\]
and monomial terms may also be defined in a similar way. To construct
these terms of polynomials we assume $\mathcal{E}$ has disjoint finite
sums, and define the negation of a map $p\colon E\to B$ to be the
map $\neg p:=\Pi_{p}\mathbf{0}$. This map may also be viewed as the
component of the distributivity pullback below
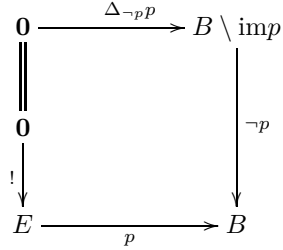
\[
\xymatrix{\mathbf{0}\ar[rr]^{\Delta_{\neg p}p}\ar@{=}[d] &  & B\setminus\textnormal{im}p\ar[dd]^{\neg p}\\
\mathbf{0}\ar[d]_{!}\\
E\ar[rr]_{p} &  & B
}
\]
This negation operator then allows us to construct the order $n$
homogeneous terms of polynomials. The simplest example of the negation
operator yielding terms of polynomials is that the \emph{constant
term} of a polynomial $\left(s,p,t\right)$ is given by
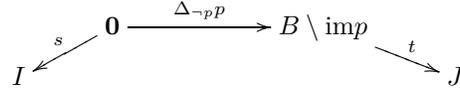
\[
\xymatrix@R=0.5em{ & \mathbf{0}\myar{\Delta_{\neg p}p}{rr}\ar[ld]_{s} &  & B\setminus\textnormal{im}p\myar{t}{rd}\\
I &  &  &  & J
}
\]

We will also show that this negation operator yields a dense-closed
orthogonal factorization system on any locally cartesian closed $\mathcal{E}$
with a strict initial object (weaker axioms that that of a topos in
which dense-closed factorizations are normally considered). 

We then construct the localization of $\mathcal{E}$ by the class
$\mathcal{W}$ of dense monomorphisms, showing that $\mathcal{E}\left[\mathcal{W}^{-1}\right]$
is locally cartesian closed, and that in this localization $\mathcal{E}\left[\mathcal{W}^{-1}\right]$
the morphisms satisfy certain decidability conditions. 

Taking $I=J=\mathbf{1}$ for simplicity, we see that this localization
has the advantage that derivatives of all polynomial functors exist.
Indeed for a polynomial
\[
\xymatrix{\mathbf{1} & E\ar[r]^{p}\ar[l] & B\ar[r] & \mathbf{1}}
\]
 the derivative is given by the polynomial
\[
\xymatrix@R=0.5em{ & E\boxtimes_{B}E\ar[r]^{\neg\delta_{p}}\ar[ld] & E\times_{B}E\myar{\pi_{1}}{r} & E\ar[dr]\\
\mathbf{1} &  &  &  & \mathbf{1}
}
\]
where we have taken the negation of the diagonal $\delta_{p}$. The
notion of the derivative of a polynomial functor is due to McBride
\cite{McBride}, however that was written for computer scientists
(and did not use the localization). In this paper we express all definitions,
proofs and results purely categorically, thus making this theory more
accessible for category theorists.

\section{Background}

In this section we recall the basic facts about distributivity pullbacks
and extensive categories. These facts will be widely used throughout
the paper.

\subsection{Distributivity pullbacks}

In order to study polynomial functors in a category $\mathcal{E}$,
we will require that $\mathcal{E}$ has some structure. First and
foremost, we will require the following property on $\mathcal{E}$.
\begin{defn}
A category $\mathcal{E}$ with pullbacks is said to be \emph{locally
cartesian closed} if for every morphism $p\colon E\to B$ in $\mathcal{E}$,
the 'pullback along $p$' functor $\Delta_{p}\colon\mathcal{E}/B\to\mathcal{E}/E$
has a right adjoint denoted $\Pi_{p}\colon\mathcal{E}/E\to\mathcal{E}/B$.
\end{defn}

Note, as in the thesis of von Glen \cite{glehn}, that a category
is locally cartesian closed when one has a pseudo-distributive law
between the pseudomonads for fibrations with sums and products. Given
that the existence of the following type of pullbacks is equivalent
to the existence of this pseudo-distributive law, is is reasonable
to use the name ``distributivity pullbacks'' given by Weber \cite{weber}.
\begin{defn}
Suppose we are given morphisms $p\colon E\to B$ and $u\colon A\to E$
in $\mathcal{E}$, the distributivity pullback (dpb) around $u$ and
$p$ is defined as the terminal pullback around $u$ and $p$, meaning
that given any other pullback around $u$ and $p$ as on the left
below

\[
\xymatrix{X'\ar[rr]\ar[dd]\ar@{}[rdddr]|-{\textnormal{pb}} &  & Y'\ar[ddd] &  & X'\ar[rr]\ar@{..>}[d]\ar@{}[rrd]|-{\textnormal{pb}} &  & Y'\ar@{..>}[d]\\
 &  &  & \ar@{}[d]|-{=} & X\ar[rr]\ar[d]\ar@{}[rrdd]|-{\textnormal{dpb}} &  & Y\ar[dd]\\
A\ar[d]_{u} &  &  & \; & A\ar[d]_{u}\\
E\ar[rr]_{p} &  & B &  & E\ar[rr]_{p} &  & B
}
\]
there exists a unique pullback square as on the right above yielding
a factorization through the distributivity pullback. 
\end{defn}

\begin{thm}
\cite{weber} A category $\mathcal{E}$ with pullbacks is locally
cartesian closed if and only if all it has all distributivity pullbacks.
\end{thm}

We will now recall a number of useful facts about distributivity pullbacks.
The first explains the above theorem.
\begin{prop}
\cite{weber} In a locally cartesian closed category $\mathcal{E}$,
the distributivity pullback around $u$ and $p$ is given by 
\[
\xymatrix{X\ar[rr]\ar[d]_{\epsilon_{u}}\ar@{}[rrdd]|-{\textnormal{dpb}} &  & Y\ar[dd]^{\Pi_{p}u}\\
A\ar[d]_{u}\\
E\ar[rr]_{p} &  & B
}
\]
where $\Pi_{p}$ is the right adjoint to pulling back along $p$,
and $\epsilon_{u}\colon\Delta_{p}\Pi_{p}u\to u$ is the counit component.
\end{prop}

The following is the horizontal pasting lemma for distributivity pullbacks.
There is also a vertical lemma which is not used in this paper.
\begin{prop}
\cite{weber} In a locally cartesian closed category $\mathcal{E}$,
the distributivity pullback around $u$ and the composite $p_{2}\cdot p_{1}$
is constructed as as the diagram 
\[
\xymatrix{X\ar@{}[rd]|-{\textnormal{pb}}\ar[d]\ar[r] & Y\ar@{}[rddd]|-{\textnormal{dpb}}\ar[r]\ar[d] & Z\ar[ddd]\\
S\ar[r]\ar@{}[rdd]|-{\textnormal{dpb}}\ar[d] & T\ar[dd]\\
A\ar[d]_{u}\\
E\ar[r]_{p_{1}} & M\ar[r]_{p_{2}} & B
}
\]
\end{prop}

Finally, it is worth pointing out that distributivity pullbacks are
well behaved with respect to monomorphisms; as shown by the following
two results.
\begin{prop}
Let $\mathcal{E}$ be locally cartesian closed. Suppose we are given
a distributivity pullback diagram 
\[
\xymatrix{X\ar[rr]\ar[d]\ar@{}[rrdd]|-{\textnormal{dpb}} &  & Y\ar[dd]^{\Pi_{p}u}\\
A\ar[d]_{u}\\
E\ar[rr]_{p} &  & B
}
\]
it follows that:
\begin{enumerate}
\item if $u\colon A\to B$ is mono, then $\Pi_{p}u\colon Y\to B$ is also
mono;
\item if $p\colon E\to B$ is mono, then $\epsilon_{u}\colon X\to A$ is
invertible.
\end{enumerate}
\end{prop}

\begin{proof}
The first is since $\Pi_{p}$ is a right adjoint (and being mono means
the kernel pair pullback consists of two identities). The second fact
is shown in \cite[Corollary 47]{mres}.
\end{proof}

\subsection{Bicategories of polynomials}

We first recall the notion of polynomial (and the associated polynomial
functor) in a general locally cartesian category $\mathcal{E}$.
\begin{defn}
Let $\mathcal{E}$ be a locally cartesian closed category. A \emph{polynomial}
in $\mathcal{E}$ from $I$ to $J$ is a diagram of the form
\[
\xymatrix{ & E\ar[r]^{p}\ar[ld]_{s} & B\ar[rd]^{t}\\
I &  &  & J
}
\]
and its \emph{associated polynomial functor }is the composite
\[
\xymatrix{\mathcal{E}/I\ar[r]^{\Delta_{s}} & \mathcal{E}/E\ar[r]^{\Pi_{p}} & \mathcal{E}/B\ar[r]^{\Sigma_{t}} & \mathcal{E}/J}
\]
A \emph{cartesian map of polynomials} is a diagram as below where
the middle square is a pullback
\[
\xymatrix{ & E_{1}\ar[r]^{p_{1}}\ar[ld]_{s_{1}}\ar[dd]\ar@{}[rdd]|-{\textnormal{pb}} & B_{1}\ar[rd]^{t_{1}}\ar[dd]\\
I &  &  & J\\
 & E_{2}\ar[r]_{p_{2}}\ar[ul]^{s_{2}} & B_{2}\ar[ur]_{t_{2}}
}
\]
\end{defn}

Much like distributivity pullbacks, the composite of two polynomials
is defined as a certain terminal diagram (with respect to placing
a pullback on top of the diagram). This allows polynomials in $\mathcal{E}$
to be assembled into a bicategory $\mathbf{Poly}\left(\mathcal{E}\right).$
\begin{defn}
\cite{weber} Let $\mathcal{E}$ be a locally cartesian closed category.
The composite of two polynomials $\left(s,p,t\right)$ and $\left(u,q,v\right)$\footnote{Equipped with a choice of a factorization of the middle map, which
is unique up to unique isomorphism as it is determined as a pullback.} is defined as the terminal diagram of the form below
\[
\xymatrix{ & X\ar[d]\ar[rr] & \ar@{}[ld]|-{\textnormal{pb}} & Y\ar[rd]\ar[rr]\ar[ld] & \ar@{}[dr]|-{\textnormal{pb}} & Z\ar[d]\\
 & E\ar[r]^{p}\ar[ld]_{s} & B\ar[rd]^{t} &  & S\ar[r]^{q}\ar[ld]_{u} & T\ar[rd]^{v}\\
I &  &  & J &  &  & K
}
\]
meaning that any other such diagram uniquely factors as the above
diagram with a pullback placed on top of the map $X\to Y\to Z$. Note
the middle square need not be a pullback above.
\end{defn}

This composite can be constructed directly using distributivity pullbacks.
\begin{prop}
\cite{weber,gambinokock} Let $\mathcal{E}$ be a locally cartesian
closed category. The composite of two polynomials $\left(s,p,t\right)$
and $\left(u,q,v\right)$ is constructed by forming the middle pullback,
right distributivity pullback, and then left pullback as below
\[
\xymatrix{ & X\ar[rr]\ar[dd] &  & Y\ar[rr]\ar[d] &  & Z\ar[dd]\\
 &  & \ar@{}[]|-{\textnormal{pb}} & M\ar[rd]\ar[ld]\ar@{}[dd]|-{\textnormal{pb}} & \ar@{}[]|-{\textnormal{dpb}}\\
 & E\ar[r]^{p}\ar[ld]_{s} & B\ar[rd]^{t} &  & S\ar[r]^{q}\ar[ld]_{u} & T\ar[rd]^{v}\\
I &  &  & J &  &  & K
}
\]
\end{prop}

\begin{rem}
We observe (at least with \emph{composites of polynomials}) that constructions
on polynomial functors are defined as terminal diagrams, and these
may be calculated explicitly by using distributivity pullbacks. Later
on, we will see that \emph{terms of polynomials} also follow this
pattern.
\end{rem}

\subsection{Extensive categories}

We now recall the notions of distributive, extensive, and lextensive
categories. These are categories in which coproducts behave nicely
with respect to certain limits such as products and pullbacks.
\begin{defn}
A category $\mathcal{E}$ with finite products and finite coproducts
is said to be\emph{ distributive }if the canonical morphism
\[
X\times Y+X\times Z\to X\times\left(Y+Z\right)
\]
is invertible for all $X,Y,Z$ in $\mathcal{E}$.
\end{defn}

The following facts about distributive categories will be used often
in this paper.
\begin{prop}
\cite{extensive} In a distributive category $\mathcal{E}$, coproduct
injections are monic and the initial object is strict.
\end{prop}

Extensive categories are those in which coproducts interact well with
pullbacks.
\begin{defn}
\cite{extensive} A category $\mathcal{E}$ with finite coproducts
is said to be\emph{ extensive }if the canonical coproduct functor
\[
\mathcal{E}/X\times\mathcal{E}/Y\to\mathcal{E}/\left(X+Y\right)
\]
is an equivalence of categories.
\end{defn}

When computing pullbacks the following is often a more useful characterization.
\begin{prop}
\cite{extensive} A category $\mathcal{E}$ with pullbacks and finite
coproducts is extensive precisely when for every commuting diagram
as below
\[
\xymatrix{A\ar[rr]\ar[d] &  & M\ar[d] &  & B\ar[ll]\ar[d]\\
X\ar[rr]_{i_{X}} &  & X+Y &  & Y\ar[ll]^{i_{Y}}
}
\]
where $i_{X}$ and $i_{Y}$ are coproduct injections, the two squares
are pullbacks if and only if the top row is a coproduct diagram.
\end{prop}

It is also worth noting that finite coproducts are disjoint in an
extensive category \cite{extensive}. Finally, we note that one often
uses a slightly stronger version of this definition which incorporates
finite limits.
\begin{defn}
\cite{extensive} An extensive category $\mathcal{E}$ is said to
be \emph{lextensive }if it has all finite limits.
\end{defn}

\section{Negation and the dense-closed factorization}

In this section we define (for each object $B$ in a locally cartesian
closed $\mathcal{E}$ with a strict initial object) a negation operator
$\neg\colon\mathcal{E}/B\to\mathcal{E}/B$. This negation operator,
defined in terms of distributivity pullbacks, is then used to define
the dense and closed morphisms in $\mathcal{E}$, and we show these
two classes of maps yield an orthogonal factorization system on $\mathcal{E}$.
In the case of $\mathbf{Set}$ this is the usual image factorization;
more generally one may view this as the 'closure of the image'.

We then go on to construct the localization of $\mathcal{E}$ by the
class $\mathcal{W}$ of dense monomorphisms, showing that it is locally
cartesian closed, and moreover that in this localization every morphism
is ``decidable'' in a suitable sense. 

\subsection{Negation operator}

The following defines the negation (image complement) of general arrow
$p\colon E\to B$. The reader may consider the case where $p$ is
a injection in $\mathbf{Set}$, in which case the image complement
is simply $B\setminus E$.
\begin{defn}
Suppose $\mathcal{E}$ is a category with a strict initial object
$\mathbf{0}$. For a given exponentiable arrow $p\colon E\to B$ we
define the \emph{negation} of $p$, denoted $\neg p\colon B\setminus\textnormal{im}p\to B$,
to be the arrow
\[
\xymatrix{\mathbf{0}\ar[rr]^{!}\ar@{=}[d] &  & B\setminus\textnormal{im}p\ar[dd]^{\neg p}\\
\mathbf{0}\ar[d]_{!}\\
E\ar[rr]_{p} &  & B
}
\]
in the distributivity pullback above. 
\end{defn}

\begin{rem}
Note that for a monomorphism $p$, $E$ is disjoint to $B\setminus E$
(defined as above with $\textnormal{im}p\cong E$) by definition.
Moreover, $B\setminus E+E\cong B$ when this is a pushout diagram.
This is always true in $\mathbf{Set}$, that is any such $p$ is ``decidable''.
We discuss decidability further later on.
\end{rem}

\subsection{The object of distinct pairs}

As a example of using this negation operator, we construct the object
of \emph{distinct} pairs $E\boxtimes_{B}E$ (as a subobject of the
kernel pair $E\times_{B}E$). In $\mathbf{Set}$, and for a function
$p\colon E\to B$, this is given by 
\[
E\boxtimes_{B}E=\left\{ \left(e_{1},e_{2}\right)\mid e_{1}\neq e_{2}\wedge pe_{1}=pe_{2}\right\} 
\]

For a general locally cartesian closed category $\mathcal{E}$ with
a strict initial, we may define this object as follows.
\begin{defn}
For a map $p\colon E\to B$ we define the \emph{object of distinct
pairs} $E\boxtimes_{B}E$ by the distributivity pullback

\[
\xymatrix{\mathbf{0}\ar[rr]^{!}\ar@{=}[d] &  & E\boxtimes_{B}E\ar[dd]^{\neg\delta_{p}}\\
\mathbf{0}\ar[d]_{!}\\
E\ar[rr]_{\delta_{p}} &  & E\times_{B}E
}
\]
where $\delta_{p}\colon E\to E\times_{B}E$ is the diagonal.
\end{defn}

\begin{rem}
Asking that this is a pushout, so that $E+E\boxtimes_{B}E\cong E\times_{B}E$,
corresponds to asking that equality is decidable: that is for any
$x,y\in E$ we have $x=y$ or $x\neq y$.
\end{rem}

\begin{rem}
More generally, we have the $n$-diagonal into the $n$-kernel pair
$\delta_{p}^{n}\colon E\to E\times_{B}E\cdots\times_{B}E$. We denote
this $n$-pullback $E^{\times n}$ and we denote $E^{\boxtimes_{B}n}\to E^{\times n}$
the negation of the $n$-diagonal. This gives us the object of distinct
$n$-tuples $E^{\boxtimes_{B}n}$.
\end{rem}

\subsection{The dense-closed OFS}

We now use the negation operator to define the dense and closed maps
in $\mathcal{E}$, and show that this is a stable orthogonal factorization
system.
\begin{defn}
Let $\mathcal{E}$ be a locally cartesian closed category with a strict
initial object. For a morphism $f\colon A\to B$ we say 
\begin{itemize}
\item $f$ is \emph{closed} iff $\neg\neg f\cong f$ in $\mathcal{E}/B$;
\item $f$ is \emph{dense} iff $\neg\neg f\cong\textnormal{id}_{B}$ in
$\mathcal{E}/B$;
\end{itemize}
\end{defn}

We must first check that we indeed have a dense-closed factorization.
\begin{thm}
\label{facexist} Every map $p\colon E\to B$ factors as a dense map
followed by a closed map. 
\end{thm}

We divide the proof into a few lemmata.
\begin{lem}
Every map $p\colon E\to B$ factors as a candidate dense map $\overline{p}$
followed by the candidate closed map $\neg\neg p$. 
\end{lem}

\begin{proof}
Since $p$ is disjoint to $\neg p$, and $\neg p$ is disjoint to
$\neg\neg p$, we have induced maps into the distributivity pullback
as below
\[
\xymatrix{\mathbf{0}\ar@{=}[d]\ar[rr]^{!} &  & E\ar@{..>}[d]^{\overline{p}}\ar@/^{2pc}/[ddd]^{p}\\
\mathbf{0}\ar[rr]^{!}\ar@{=}[d] &  & \textnormal{im}p\ar[dd]^{\neg\neg p}\\
\mathbf{0}\ar[d]_{!} & \textnormal{dpb}\\
B\setminus\textnormal{im}p\ar[rr]_{\neg p} &  & B
}
\]
\end{proof}
\begin{rem}
The map $\overline{p}\colon p\to\neg\neg p$ is as appears in intuitionistic
logic \cite{martinlof}, which is to be expected in a locally cartesian
closed $\mathcal{E}$.
\end{rem}

It remains to check that $\overline{p}$ is dense and that $\neg\neg p$
is closed; but we first need the following.
\begin{lem}
\label{tplneg} The triple negation $\neg\neg\neg p$ is isomorphic
to the single negation $\neg p$.
\end{lem}

\begin{proof}
To see this we note the existence of comparisons $\theta$ and $\phi$
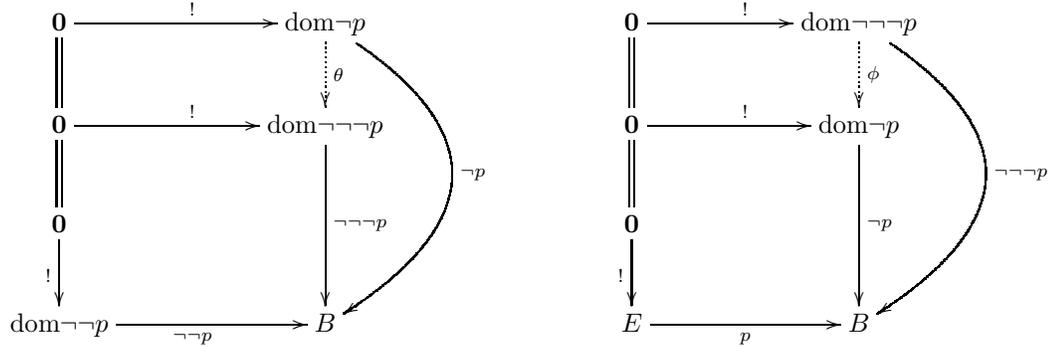
\[
\xymatrix{\mathbf{0}\ar@{=}[d]\ar[rr]^{!} &  & \textnormal{dom}\neg p\ar@{..>}[d]^{\theta}\ar@/^{4pc}/[ddd]^{\neg p} &  &  & \mathbf{0}\ar@{=}[d]\ar[rr]^{!} &  & \textnormal{dom}\neg\neg\neg p\ar@{..>}[d]^{\phi}\ar@/^{4pc}/[ddd]^{\neg\neg\neg p}\\
\mathbf{0}\ar[rr]^{!}\ar@{=}[d] &  & \textnormal{dom}\neg\neg\neg p\ar[dd]^{\neg\neg\neg p} &  &  & \mathbf{0}\ar[rr]^{!}\ar@{=}[d] &  & \textnormal{dom}\neg p\ar[dd]^{\neg p}\\
\mathbf{0}\ar[d]_{!} &  &  &  &  & \mathbf{0}\ar[d]_{!}\\
\textnormal{dom}\neg\neg p\ar[rr]_{\neg\neg p} &  & B &  &  & E\ar[rr]_{p} &  & B
}
\]
where the disjointness of $p$ and $\neg\neg\neg p$ is shown in the
pullback diagram below
\[
\xymatrix{\mathbf{0}\ar[dd]\ar@{=}[rr] &  & \mathbf{0}\ar[rr]^{!}\ar@{=}[d] &  & \textnormal{dom}\neg\neg\neg p\ar[dd]^{\neg\neg\neg p}\\
 &  & \mathbf{0}\ar[d]_{!}\\
E\ar[rr]_{\overline{p}}\ar@/_{2pc}/[rrrr]_{p} &  & \textnormal{im}p\ar[rr]_{\neg\neg p} &  & B
}
\]
\end{proof}
\begin{rem}
It almost immediately follows that the double complement $\neg\neg$
defines a closure operator (idempotent monad on each slice category
of monomorphisms $\mathcal{E}_{\textnormal{mono}}/B$).
\end{rem}

\begin{proof}
We can now finish the proof of Theorem \ref{facexist}.
\end{proof}
\begin{lem}
We have that $\overline{p}$ is dense and that $\neg\neg p$ is closed;
\end{lem}

\begin{proof}
That $\neg\neg p$ is closed is an immediate consequence. Now to show
that $\overline{p}$ is dense meaning $\neg\neg\overline{p}\cong\textnormal{id}$
(that is $\neg\neg\neg\overline{p}\cong\neg\overline{p}\cong0$ by
the above), we use that:
\begin{itemize}
\item the negation of $p$, constructed as a dpb, may be constructed as
the horizontal composite of two dpbs if we decompose $p$ as $\neg\neg p\cdot\overline{p}$
\item $\neg\neg p$ is mono, making the middle counit invertible
\end{itemize}
thus giving the diagram

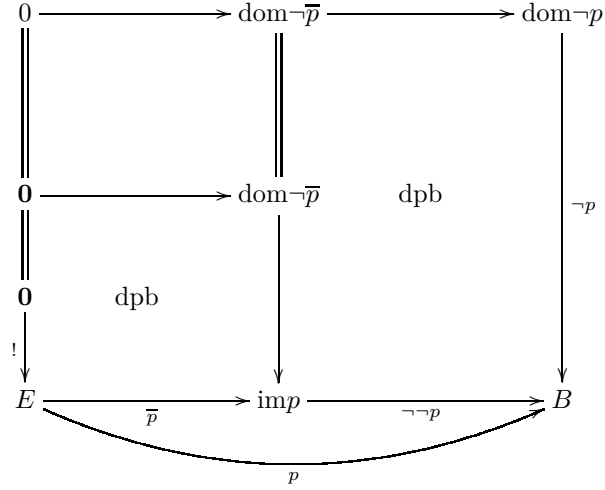
\[
\xymatrix{0\ar[rr]\ar@{=}[dd] &  & \textnormal{dom}\neg\overline{p}\ar@{=}[dd]\ar[rr] &  & \textnormal{dom}\neg p\ar[dddd]^{\neg p}\\
\\
\mathbf{0}\ar@{=}[d]\ar[rr] &  & \textnormal{dom}\neg\overline{p}\ar[dd] & \textnormal{dpb}\\
\mathbf{0}\ar[d]_{!} & \textnormal{dpb}\\
E\ar[rr]_{\overline{p}}\ar@/_{2pc}/[rrrr]_{p} &  & \textnormal{im}p\ar[rr]_{\neg\neg p} &  & B
}
\]
and we notice that the pullback of $\neg p$ and $\neg\neg p$ is
$0$, so that $\textnormal{dom}\neg\overline{p}=0$ and hence $\neg\overline{p}$
is the initial map.
\end{proof}
We recall the notion of an orthogonal factorization system.
\begin{defn}
An orthogonal factorization system (OFS) on a category $\mathcal{E}$
consists of two classes of morphisms $\mathscr{L}$ and $\mathscr{R}$
such that
\begin{itemize}
\item every map $f$ decomposes as $f=r\cdot l$ with $r\in\mathscr{R}$
and $l\in\mathscr{L}$
\item given any other decomposition $f=r'\cdot l'$, there exists a unique
isomorphism $\alpha$ rendering commutative
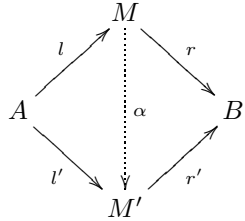
\[
\xymatrix{ & M\ar[rd]^{r}\ar@{..>}[dd]^{\alpha}\\
A\ar[ru]^{l}\ar[rd]_{l'} &  & B\\
 & M'\ar[ru]_{r'}
}
\]
\item $\mathscr{L}$ and $\mathscr{R}$ are closed under composition and
contain all isomorphisms
\end{itemize}
\end{defn}

\begin{thm}
Let $\mathcal{E}$ be a locally cartesian closed category with a strict
initial object. Then $\left(\mathsf{dense},\mathsf{closed}\right)$
defines an orthogonal factorization system on $\mathcal{E}$.
\end{thm}

\begin{proof}
We need only check that the factorization is unique (that is, it is
necessarily the same as that given earlier), so let us suppose we
can factor $p=c\cdot d$ where $c$ is closed and $d$ is dense. We
construct $\neg p$ as the composite of two dpbs 
\[
\xymatrix{0\ar@{=}[rr]\ar@{=}[d] &  & 0\ar@{=}[d]\ar[rr] &  & B\setminus\textnormal{im}p\ar[ddd]^{\neg p}\\
0\ar@{=}[d]\ar@{=}[rr] &  & 0\ar[dd]\\
0\ar[d] & \textnormal{dpb} &  & \textnormal{dpb}\\
E\ar[rr]_{d} &  & M\ar[rr]_{c} &  & B
}
\]
This means that $\neg p\cong\neg c$ that is $\neg\neg p\cong\neg\neg c\cong c$,
and since $c$ is mono we see that $d\cong\overline{p}$ (accounting
for the isomorphism in the slice category).
\end{proof}
\begin{prop}
The factorization $\left(\mathsf{dense},\mathsf{closed}\right)$ is
pullback stable.
\end{prop}

\begin{proof}
Suppose we are given a dense map $d$ which pulls back into a map
$d'$. We can see that $d'$ is also dense by constructing the diagram
on the left
\[
\xymatrix{0\ar@{=}[d]\ar[rr] &  & \textnormal{dom}\neg d'\ar[dd]^{\neg d'} &  & 0\ar[rr]\ar@{..>}[d] &  & \textnormal{dom}\neg d'\ar@{..>}[d]\\
0\ar[d] & \textnormal{dpb} &  &  & 0\ar@{=}[d]\ar[rr] &  & \textnormal{dom}\neg d\ar[dd]_{\neg d}\\
\bullet\ar[rr]_{d'}\ar[d] &  & \bullet\ar[d] & = & 0\ar[d] & \textnormal{dpb}\\
\bullet\ar[rr]_{d} &  & \bullet &  & \bullet\ar[rr]_{d} &  & \bullet
}
\]
giving induced maps into the dpb on the right above. As $\textnormal{dom}\neg d\cong0$,
and the initial is strict, we have that $\neg d'$ is also the initial
map.
\end{proof}

\subsection{Lawvere-Tierney axioms}

It is worth pointing out that the double negation (seen as a closure)
satisfies all of the Lawvere-Tierney axioms one would expect, with
the exception that we are only dealing with a locally cartesian closed
category with a strict initial object. Hence this gives examples of
Lawvere-Tierney topologies which are not necessarily a topos.

In the following proposition we restrict to monomorphisms to make
the axioms as similar to those of a Lawvere-Tierney topology as possible.
\begin{prop}
Suppose $\mathcal{E}$ is a locally cartesian closed category with
a strict initial object. Suppose $A\to X$ and $B\to X$ are monomorphisms.
Define the closure of $A$ (as a subobject of $X$) as $\overline{A}:=\neg\neg A$.
Then
\begin{itemize}
\item there exists a monomorphism $A\to\overline{A}$ in $\mathcal{E}/X$;
\item we have $\overline{\overline{A}}\cong\overline{A}$
\item we have $\overline{A}\cap\overline{B}\cong\overline{A\cap B}$;
\item closed maps are pullback stable;
\item we have an assignation sending any monomorphism $A\to B$ in $\mathcal{E}/X$
to a monomorphism $\overline{A}\to\overline{B}$ in $\mathcal{E}/X$
\end{itemize}
\end{prop}

\begin{proof}
The map $A\to\overline{A}$ is simply the first component of the dense-closed
factorization. That $\overline{\overline{A}}\cong\overline{A}$ follows
from the triple negation being isomorphic to the single negation as
shown in Lemma \ref{tplneg}. Closed maps are pullback stable as these
are the right class of a factorization system. To show that $\overline{A}\cap\overline{B}=\overline{A\cap B}$
we construct the diagram

\[
\xymatrix{A\cap B\ar[rr]\ar[rdrd]\ar[dd] &  & \bullet\ar[dd]\ar[rr] &  & B\ar[dd]\\
\\
\bullet\ar[rr]\ar[dd] &  & \overline{A}\cap\overline{B}\ar[rrdd]\ar[rr]\ar[dd] &  & \overline{B}\ar[dd]\\
\\
A\ar[rr] &  & \overline{A}\ar[rr] &  & X
}
\]
and note the diagonal composite is a dense-closed factorization, since
the dense and closed maps are pullback stable. As the factorization
is unique we may make the identification $\overline{A}\cap\overline{B}\cong\overline{A\cap B}$.
Finally, given any map $A\to B\to X$ equal to $A\to X$ (that is
a map in the slice, also since $B\to X$ and $A\to X$ are monos,
so is $A\to B$) we recover a map $\neg B\to\neg A$ as on the left
below
\[
\xymatrix{0\ar[rrrr]\ar[rdd] &  &  &  & \neg B\ar[dl]\ar@/^{2pc}/[dddl]\\
 & 0\ar@{=}[d]\ar[rr] &  & \neg A\ar[dd] &  & 0\ar@{=}[rr]\ar[d] &  & 0\ar[rr]\ar[d] &  & \neg B\ar[d]\\
 & 0\ar[d] &  &  &  & A\ar[rr] &  & B\ar[rr] &  & X\\
 & A\ar[rr] &  & X
}
\]
where we used that $\neg B$ is as disjoint from $A$ as shown on
the right above. Repeating this process once more we recover the map
$\overline{A}\to\overline{B}$ between the closures.
\end{proof}

\subsection{Decidability as a dense monomorphism}

The following defines the notion of decidability we will use in this
paper. This is required later for derivatives.
\begin{defn}
\label{decidable} Suppose $\mathcal{E}$ is a locally cartesian closed
category with disjoint finite sums. Given a monomorphism $A\to X$
and its negation denoted as a map $\neg A\to X$, we call the induced
map $A+\neg A\to X$ the \emph{decidability map} of $A\to X$. We
say that $A\to X$ is \emph{decidable} when this decidability map
is invertible.
\end{defn}

\begin{prop}
Every decidability map $A+\neg A\to X$ is a dense monomorphism.
\end{prop}

\begin{proof}
Consider a monomorphism $A\to X$, and its negation 
\[
\xymatrix{0\ar[rr]\ar@{=}[d] &  & \neg A\ar[dd]\\
0\ar[d]\\
A\ar[rr] &  & X
}
\]
as in the distributivity pullback above. The decidability map $A+\neg A\to X$
is exhibited as a monomorphism by summing the four pullbacks
\[
\xymatrix{0\ar[rr]\ar[dd] &  & \neg A\ar[dd] & A\ar@{=}[rr]\ar@{=}[dd] &  & A\ar[dd]\\
\\
A\ar[rr] &  & X & A\ar[rr] &  & X
}
\]
and
\[
\xymatrix{\neg A\ar@{=}[rr]\ar@{=}[dd] &  & \neg A\ar[dd] & 0\ar[rr]\ar[dd] &  & A\ar[dd]\\
\\
\neg A\ar[rr] &  & X & \neg A\ar[rr] &  & X
}
\]
using that $\mathcal{E}$ is extensive. To check the map $A+\neg A\to X$
is dense, we consider the dpb
\[
\xymatrix{0\ar[rr]\ar@{=}[d] &  & P\ar[dd]\\
0\ar[d]\\
A+\neg A\ar[rr] &  & X
}
\]
and note we can restrict the pullback along both inclusions from $A$
and $\neg A$ as below
\[
\xymatrix{0\ar[dd]\ar@{=}[rr] &  & 0\ar[rr]\ar@{=}[d] &  & P\ar[dd] & 0\ar[dd]\ar@{=}[rr] &  & 0\ar[rr]\ar@{=}[d] &  & P\ar[dd]\\
 &  & 0\ar[d] &  &  &  &  & 0\ar[d]\\
A\ar[rr] &  & A+\neg A\ar[rr] &  & X & \neg A\ar[rr] &  & A+\neg A\ar[rr] &  & X
}
\]
By definition of $\neg A$ and $\neg\neg A$ in terms distributivity
pullbacks, we have induced maps $P\to\neg A$ and $P\to\neg\neg A$
respectively from these two outside pullbacks above. Finally to see
that $P$ is initial, note that we have the induced map into the pullback
\[
\xymatrix{P\ar@{..>}[rd]\ar[rdrr]\ar[rddd]\\
 & 0\ar[dd]\ar[rr] &  & \neg\neg A\ar[dd]\\
\\
 & \neg A\ar[rr] &  & X
}
\]
and this map $P\to0$ exhibits $P$ as initial by strictness.
\end{proof}
\begin{example}
Consider the diagonal map $E\to E\times E$ which is a monomorphism.
The negation $E\boxtimes E\to E$ is the object of distinct pairs.
This $E\to E\times E$ is decidable when equality is decidable, that
is $x=y$ or $x\neq y$ for all $x,y\in E$.
\end{example}

\subsection{Dense-mono localization}

In this section we construct the localization of $\mathcal{E}$ by
the class $\mathcal{W}$ of dense monomorphisms. We then show that
this localization $\mathcal{E}\left[\mathcal{W}^{-1}\right]$ inherits
local cartesian closedness from $\mathcal{E}$.
\begin{prop}
Let $\mathcal{E}$ be a locally cartesian closed category with a strict
initial object. Then the localization $\mathcal{E}\left[\mathcal{W}^{-1}\right]$
by the class of $\mathcal{W}$ dense monomorphisms is a calculus of
right fractions, and is equivalent to the category with objects of
$\mathcal{E}$ and morphisms $X\nrightarrow Y$ given by spans
\[
\xymatrix{ & M\ar[rd]^{f}\ar[ld]_{w}\\
X &  & Y
}
\]
with $w\colon M\to X$ a dense monomorphism and $f\colon M\to Y$
any morphism of $\mathcal{E}$. Composition is by pullback, and two
spans $\left(f,w\right)$ and $\left(f',w'\right)$ are identified
when there exists a commuting diagram

\[
\xymatrix{ & M\ar[rd]^{f}\ar[ld]_{w}\\
X & T\ar[r]\ar[l]\ar[d]\ar[u] & Y\\
 & M'\ar[ru]_{w'}\ar[ul]^{f'}
}
\]
\end{prop}

\begin{proof}
It is not hard to check this satisfies the axioms needed for a calculus
of right fractions, given that the dense monomorphisms are pullback
stable.

As this localization is a calculus of right fractions, the functor
$\mathcal{E}\to\mathcal{E}\left[\mathcal{W}^{-1}\right]$ is left
exact (preserves finite limits). In fact, more is true: this functor
preserves distributivity pullbacks. This is evident from the following
proposition.
\end{proof}
\begin{prop}
Let $\mathcal{E}$ be a locally cartesian closed category with a strict
initial object. Then the localization $\mathcal{E}\left[\mathcal{W}^{-1}\right]$
is locally cartesian closed.
\end{prop}

\begin{proof}
We will prove this by describing how to form the distributivity pullback
in the localization. First note that a pullback of two spans $\left(f,w\right)$
and $\left(z,g\right)$ is given as below
\[
\xymatrix{\bullet\ar[r]^{\left(1,f'z\right)}\ar[d]_{\left(1,g'w\right)} & \bullet\ar[d]^{\left(z,g\right)} &  &  & \bullet\ar[r]^{f'}\ar[d]_{g'} & \bullet\ar[d]^{g}\\
\bullet\ar[r]_{\left(f,w\right)} & \bullet &  &  & \bullet\ar[r]_{f} & \bullet
}
\]
where $f'$ and $g'$ are constructed as the pullback in $\mathcal{E}$
on the right above. We construct the distributivity pullback of a
pair of composable maps $\left(p,w\right)$ and $\left(u,z\right)$
as below
\[
\xymatrix{\bullet\ar[rr]^{\left(\widetilde{p},w'''z''\right)}\ar[d]_{\left(\widetilde{\widetilde{e}},1\right)} &  & \bullet\ar[dd]^{\left(\Pi_{p}u',1\right)}\\
\bullet\ar[d]_{\left(u,z\right)}\\
\bullet\ar[rr]_{\left(p,w\right)} &  & \bullet
}
\]
where the data comprising these spans are constructed as 
\[
\xymatrix{\bullet\ar[d]_{\widetilde{\widetilde{e}}=\Pi_{w''\widetilde{e}}}\ar@{}[rrd]|-{\textnormal{dpb}} &  &  & \bullet\ar[dl]_{\widetilde{e}=\Pi_{z'}e}\ar[lll]_{w'''}\\
\bullet\ar@{}[drr]|-{\textnormal{dpb}} &  & \bullet\ar[ll]_{w''=\Pi_{z}w'}\ar@{}[r]|-{\textnormal{dpb}} & \bullet\ar[r]^{\widetilde{p}}\ar[dl]^{e}\ar[u]^{z''} & \bullet\ar[ddl]^{\Pi_{p}u'}\\
\bullet\ar[d]_{u}\ar[u]^{z}\ar@{}[drr]|-{\textnormal{pb}} &  & \bullet\ar[d]_{u'}\ar[ll]_{w'}\ar[u]^{z'}\ar@{}[rd]|-{\textnormal{dpb}} & \;\\
\bullet &  & \bullet\ar[r]_{p}\ar[ll]^{w} & \bullet
}
\]
where we have used the fact that the two top distributivity pullbacks
have invertible counit maps (and thus may be written simply as squares),
as $z$ and $w''$ are monomorphisms. The verification of the universal
property is a simple but tedious exercise.
\end{proof}
Though not required to form the calculus of right fractions, the dense
monos come close to satisfying the 2-out-of-3 property in that it
holds on one side.
\begin{prop}
Suppose the composite $gf$ and $g$ are both dense mono, then $f$
is also dense mono.
\end{prop}

\begin{proof}
Suppose $gf$ is dense mono. Using strictness of the initial object,
the negation of $gf$ is given by the composite distributivity pullback
\[
\xymatrix{0\ar@{=}[d]\ar@{=}[rr] &  & 0\ar[d]\ar@{=}[rr] &  & 0\ar[ddd]\\
0\ar@{=}[d]\ar[rr] &  & B\setminus\textnormal{im}f\ar[dd]\\
0\ar[d]\\
A\ar[rr]_{f} &  & B\ar[rr]_{g} &  & C
}
\]
Further supposing $g$ is mono, the counit map $0\to B\setminus\textnormal{im}f$
is invertible, showing $f$ is dense.
\end{proof}
\begin{rem}
If $gf$ and $f$ are dense mono, then $g$ is dense since the dense
maps are a left class.
\end{rem}

\section{Terms of polynomials}

A fact that one often takes for granted in the category of sets, is
that terms of polynomials exist. For instance, every polynomial with
$I=J=\mathbf{1}$ has a corresponding polynomial functor of the form

\[
\mathcal{E}\to\mathcal{E}\colon X\mapsto\sum_{b\in B}X^{E_{b}}
\]
The coefficient of $x^{\mathbf{n}}$ for a natural number $\text{\ensuremath{\mathbf{n}}}=\left\{ 1,...,n\right\} $
is then the set of fibres of size $n$ i.e. $\left\{ E_{b}:\left|E_{b}\right|=n\right\} $.
In this section we generalize the notion of ``terms of polynomials''
from the category of sets $\mathbf{Set}$, to a general locally cartesian
closed category $\mathcal{E}$.

Perhaps surprisingly, we will see that terms of polynomials are constructed
from the negation operator of the preceding section. The simplest
example of this is the constant term of polynomial (that is the fibres
of size zero), where for a polynomial as on the left below
\[
\xymatrix{ & E\ar[ld]\ar[r]^{p} & B\ar[rd] &  & \ar@{}[d]|-{\overset{\textnormal{const}}{\mapsto}} &  & 0\ar[ld]\ar[rr]^{\Delta_{\neg p}\left(p\right)} &  & B\setminus\textnormal{im}p\ar[rd]\\
\mathbf{1} &  &  & \mathbf{1} & \; & \mathbf{1} &  &  &  & \mathbf{1}
}
\]
the constant term is given by the right above, constructed by pulling
back along negation. Using this negation, we will show how one constructs
terms of order $n$, however the existence of terms of a general order
$F\in\mathcal{E}$ is left as an open question.

\subsection{Monomials and homogeneous terms}

We firstly observe that a term of a polynomial should itself be a
polynomial. As many constructions involving polynomial functors involve
distributivity pullbacks (equivalent to local cartesian closedness)
it is not surprising that the definition of terms of polynomials looks
somewhat similar; that is, the definition involves terminal diagrams
of a certain form.
\begin{rem}
There are other examples of this phenomenon; for example it is known
that the composite of two polynomials may be defined as the terminal
diagram of a certain form; see Weber \cite{weber}.
\end{rem}

We first define the order $F$ homogeneous term of a polynomial. For
example given a many variable polynomial, involving say $x$ and $y$,
the order $3$ term may look like $x^{3}+3x^{2}y$.
\begin{defn}
We define the order $F$ \emph{homogeneous term} of a polynomial $\left(s,p,t\right)$
to be the terminal diagram of the form
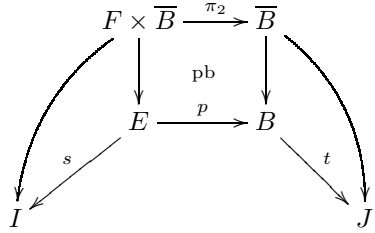
\[
\xymatrix{ & F\times\overline{B}\myar{\pi_{2}}{r}\ar[d]\ar@/_{1pc}/[dld]\ar@{}[dr]|-{\textnormal{pb}} & \overline{B}\ar[d]\ar@/^{1pc}/[ddr]\\
 & E\ar[r]^{p}\ar[ld]_{s} & B\ar[rd]^{t}\\
I &  &  & J
}
\]
where $\left(s,p,t\right)$ and $F$ are given and $\pi_{2}$ is a
product projection. This means that given any other diagram of this
form (as on the left below)
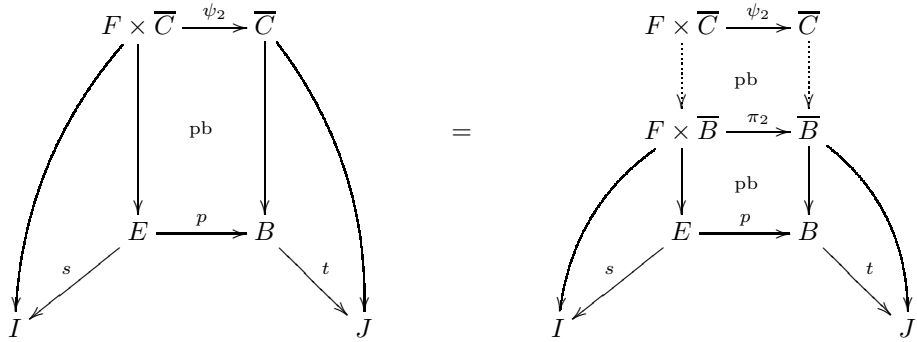
\begin{equation}
\xymatrix{ & F\times\overline{C}\myar{\psi_{2}}{r}\ar[dd]\ar@/_{1pc}/[dldd]\ar@{}[ddr]|-{\textnormal{pb}} & \overline{C}\ar[dd]\ar@/^{1pc}/[ddrd] &  &  &  & F\times\overline{C}\myar{\psi_{2}}{r}\ar@{..>}[d]\ar@{}[dr]|-{\textnormal{pb}} & \overline{C}\ar@{..>}[d]\\
 &  &  &  & = &  & F\times\overline{B}\myar{\pi_{2}}{r}\ar[d]\ar@/_{1pc}/[dld]\ar@{}[dr]|-{\textnormal{pb}} & \overline{B}\ar[d]\ar@/^{1pc}/[ddr]\\
 & E\ar[r]^{p}\ar[ld]_{s} & B\ar[rd]^{t} &  &  &  & E\ar[r]^{p}\ar[ld]_{s} & B\ar[rd]^{t}\\
I &  &  & J &  & I &  &  & J
}
\label{termpoly}
\end{equation}
we have a unique factorization as on the right above, where the induced
square is a pullback.
\end{defn}

Sometimes we wish to only consider a monomial term of a polynomial
(meaning we have just one term). To do this, we must specify the term
we are talking about. For example, we may specify we are considering
the $x^{2}y$ term of a polynomial. We note this has order $F=\mathbf{3}=\left\{ 1,2,3\right\} $,
variables $I=\left\{ x,y\right\} $, and the fact we have two instances
of $x$ and one of $y$ may be indicated by giving a function $\sigma\colon F\to I$
which sends $1$ and $2$ to $x$ and $3$ to $y$. 

In this way, the monomial term in question is specified by a map $\sigma\colon F\to I$.
The data of a such a map indicates the term's order $F$, and how
that order is decomposed into the input variables. Of course, when
$I=\mathbf{1}$ no such decomposition is necessary, and the homogeneous
terms are monomials.
\begin{defn}
We define the order $\sigma\colon F\to I$ \emph{monomial term} of
a polynomial $\left(s,p,t\right)$ to be the terminal diagram of the
form
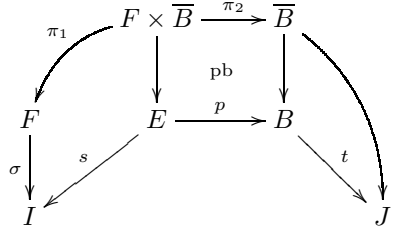
\[
\xymatrix{ & F\times\overline{B}\myar{\pi_{2}}{r}\ar[d]\ar@/_{1pc}/[dl]_{\pi_{1}}\ar@{}[dr]|-{\textnormal{pb}} & \overline{B}\ar[d]\ar@/^{1pc}/[ddr]\\
F\ar[d]_{\sigma} & E\ar[r]^{p}\ar[ld]_{s} & B\ar[rd]^{t}\\
I &  &  & J
}
\]
where $\left(s,p,t\right)$ and $\sigma\colon F\to I$ are given and
$\pi_{1}$ and $\pi_{2}$ are product projections. This means that
given any other diagram of this form (as on the left below)
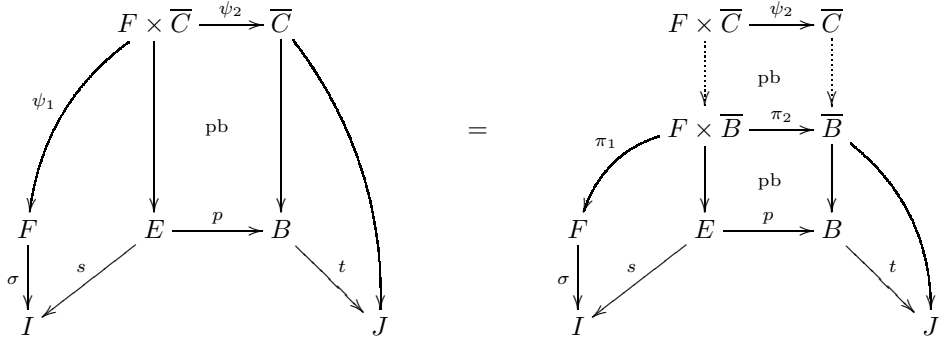
\begin{equation}
\xymatrix{ & F\times\overline{C}\myar{\psi_{2}}{r}\ar[dd]\ar@/_{1pc}/[dld]_{\psi_{1}}\ar@{}[ddr]|-{\textnormal{pb}} & \overline{C}\ar[dd]\ar@/^{1pc}/[ddrd] &  &  &  & F\times\overline{C}\myar{\psi_{2}}{r}\ar@{..>}[d]\ar@{}[dr]|-{\textnormal{pb}} & \overline{C}\ar@{..>}[d]\\
 &  &  &  & = &  & F\times\overline{B}\myar{\pi_{2}}{r}\ar[d]\ar@/_{1pc}/[ld]_{\pi_{1}}\ar@{}[dr]|-{\textnormal{pb}} & \overline{B}\ar[d]\ar@/^{1pc}/[ddr]\\
F\ar[d]_{\sigma} & E\ar[r]^{p}\ar[ld]_{s} & B\ar[rd]^{t} &  &  & F\ar[d]_{\sigma} & E\ar[r]^{p}\ar[ld]_{s} & B\ar[rd]^{t}\\
I &  &  & J &  & I &  &  & J
}
\label{termpoly-1}
\end{equation}
we have a unique factorization as on the right above, where the induced
square is a pullback.
\end{defn}

\subsection{$n$-fibre pullbacks}

In this subsection we introduce a type of pullback which is required
to construct these terms of polynomials. 
\begin{defn}
The $n$-fibre pullback $\left(n\textnormal{fpb}\right)$ about an
arrow $p$, is the terminal pullback diagram of the form

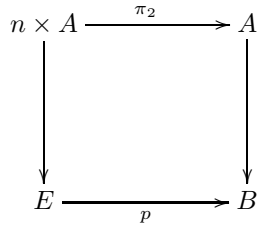
\[
\xymatrix{{\displaystyle n\times A}\ar[rr]^{\pi_{2}}\ar[dd] &  & A\ar[dd]\\
\\
E\ar[rr]_{p} &  & B
}
\]
where $n$ refers to $n$ copies of the terminal object in a locally
cartesian closed $\mathcal{E}$. This means given any other such pullback
\[
\xymatrix{{\displaystyle n\times B}\ar[rr]^{\pi_{2}}\ar[dd] &  & B\ar[dd] &  & {\displaystyle n\times B}\ar[rr]^{\pi_{2}}\ar@{..>}[d] &  & B\ar@{..>}[d]\\
 &  &  & = & {\displaystyle n\times A}\ar[rr]^{\pi_{2}}\ar[d] &  & A\ar[d]\\
E\ar[rr]_{p} &  & B &  & E\ar[rr]_{p} &  & B
}
\]
there exists a unique factorization as above where all squares are
pullbacks.
\end{defn}

\begin{rem}
Whilst infinite order terms may always be defined, and do exist in
some cases, for example the polynomial infinity case $\mathbb{N}\to\mathbf{1}$
when $\mathcal{E}=\mathbf{Set}$. It is unclear if they exist in general,
as it's unclear if the above $n$-fibre pullback can be constructed
with infinite order. One would need infinite products (actually infinite
kernel pairs) and coproducts.
\end{rem}

\begin{rem}
In the authors Master's thesis \cite{mres} the singleton (1) fibre
pullbacks were defined and constructed. Hence this may regarded as
a generalization to the $n$-ary case.
\end{rem}

This following Lemma is the first step towards constructing these
$n$-fibre pullbacks. 
\begin{lem}
Let $\mathcal{E}$ be a locally cartesian closed category with disjoint
finite coproducts. Then the map $n\cdot\delta\colon E^{\boxtimes_{B}n}\times E\to E^{\boxtimes_{B}n}$
is a monomorphism.
\end{lem}

\begin{proof}
We will prove the $n=2$ case, as the general case is similar. We
first note that we have pullbacks
\[
\xymatrix{{\displaystyle E}\ar[rr]^{\delta}\ar[dd]_{\delta} &  & E\times_{B}E\ar[dd]^{E\times\delta}\\
\\
E\times_{B}E\ar[rr]_{\left(a,b\right)\mapsto\left(a,b,a\right)} &  & E\times_{B}E\times_{B}E
}
\]
since if we label the pullback projections $f$ and $g$, we may form
the pullback
\[
\xymatrix{{\displaystyle P}\ar[rr]^{f}\ar[dd]_{g} &  & E\times_{B}E\ar[dd]^{E\times\delta}\ar[rr]^{\mapsto b} &  & E\ar[dd]^{\delta}\\
\\
E\times_{B}E\ar[rr]_{\mapsto\left(a,b,a\right)} &  & E\times_{B}E\times_{B}E\ar[rr]_{\mapsto\left(b,c\right)} &  & E\times_{B}E
}
\]
and as the bottom composite is the identity, $f\colon P\to E\times_{B}E$
is $\delta\colon E\to E\times_{B}E$, and thus $g$ is also since
$E\times\delta$ is mono. Hence we have the pullback exhibiting disjointness
\[
\xymatrix{0\ar[rr]\ar[dd] &  & E\boxtimes_{B}E\ar[dd]^{\neg\delta}\ar@/^{5pc}/[dddd]^{\mapsto\left(a,b,b\right)}\\
\\
{\displaystyle E}\ar[rr]^{\delta}\ar[dd]_{\delta} &  & E\times_{B}E\ar[dd]^{E\times\delta}\\
\\
E\times_{B}E\ar[rr]_{\mapsto\left(a,b,a\right)} &  & E\times_{B}E\times_{B}E
}
\]
Note the above argument may be generalized to $n$-ary diagonal maps.
It follows by summing pullbacks (in the lextensive category $\mathcal{E}$)
that the square
\[
\xymatrix{E\boxtimes_{B}E+E\boxtimes_{B}E\ar@{=}[rr]\ar@{=}[dd] &  & E\boxtimes_{B}E+E\boxtimes_{B}E\ar[dd]^{\left(a,b,a\right)+\left(a,b,b\right)}\\
\\
E\boxtimes_{B}E+E\boxtimes_{B}E\ar[rr]_{\left(a,b,a\right)+\left(a,b,b\right)} &  & E\boxtimes_{B}E\times_{B}E
}
\]
is a pullback square (in the $n=2$ case), and more generally the
squares below are pullbacks

\[
\xymatrix{nE^{\boxtimes_{B}n}\ar@{=}[rr]\ar@{=}[dd] &  & nE^{\boxtimes_{B}n}\ar[dd]^{\sum\delta_{j}}\\
\\
nE^{\boxtimes_{B}n}\ar[rr]_{\sum\delta_{i}} &  & E^{\boxtimes_{B}n}\times E
}
\]
as each pullback of $\delta_{i}\neq\delta_{j}$ is the initial object.
Hence $n\delta=\sum\delta_{i}$ is a monomorphism.
\end{proof}
The reason we needed to show this was a monomorphism was the confirm
the following, which follows from codiagonal maps being stable in
lextensive categories.
\begin{cor}
The pullback diagram below

\[
\xymatrix{X\ar@{=}[rr]\ar[d] &  & X\ar[rr]^{\textnormal{\ensuremath{\lambda}}}\ar[d] &  & _{p_{n}}E\ar[dd]^{p_{n}}\\
n\cdot E^{\boxtimes_{B}n}\ar@{=}[d]\ar@{=}[rr] &  & n\cdot E^{\boxtimes_{B}n}\ar[d]_{n\cdot\delta}\\
n\cdot E^{\boxtimes_{B}n}\ar[rr]^{n\cdot\delta} &  & E^{\boxtimes_{B}n}\times E\ar[rr]^{\pi} &  & E^{\boxtimes_{B}n}
}
\]
exhibits $\lambda$ as a codiagonal, so that $X\cong n\cdot{}_{p_{n}}E$.
\end{cor}

\begin{rem}
As our category is lextensive, we sometimes write the product projection
$n\cdot E\to E$ as the codiagonal $\sum_{n}E\to E$.
\end{rem}

The following takes us closer to constructing the $n$-fibre pullbacks.
\begin{lem}
\label{nfaclem} Let $\mathcal{E}$ be a locally cartesian closed
category with disjoint finite coproducts. Let $n$ be fixed. Then
the pullback as on the left below
\[
\xymatrix{n\cdot{}_{p_{n}}E\ar[rr]^{\textnormal{codiag}}\ar[d] &  & _{p_{n}}E\ar[dd]^{p_{n}} &  & n\cdot A\ar[rr]^{\textnormal{codiag}}\ar[dd]_{h'} &  & A\ar[dd]^{h}\\
n\cdot E^{\boxtimes_{B}n}\ar[d]_{n\cdot\delta}\\
E^{\boxtimes_{B}n}\times E\ar[rr]^{\pi} &  & E^{\boxtimes_{B}n} &  & E^{\boxtimes_{B}n}\times E\ar[rr]^{\pi} &  & E^{\boxtimes_{B}n}
}
\]
is terminal among pullbacks as on the right above.
\end{lem}

\begin{proof}
Suppose we are given a pullback as on the right above. We now check
that $h'$ factors through $n\cdot\delta$ (we assume without loss
of generality that only representative coproducts are used) so we
must check the pullback $x$ below is invertible. Note also that $x$
below is mono.
\[
\xymatrix{n\cdot A\ar[rr]^{x}\ar[dd] &  & n\cdot A\ar[rr]^{\textnormal{codiag}}\ar[dd]_{h'} &  & A\ar[dd]^{h}\\
\\
n\cdot E^{\boxtimes_{B}n}\ar[rr]_{n\cdot\delta} &  & E^{\boxtimes_{B}n}\cdot E\ar[rr]_{\pi} &  & E^{\boxtimes_{B}n}
}
\]
Note we used above that codiagonals are pullback stable (so $\textnormal{codiag}\cdot x=\textnormal{codiag})$.
Now, to see $x$ is invertible, it suffices to check each component
of $x$, which we call $x_{m}$ (and is also a mono), is invertible
as below
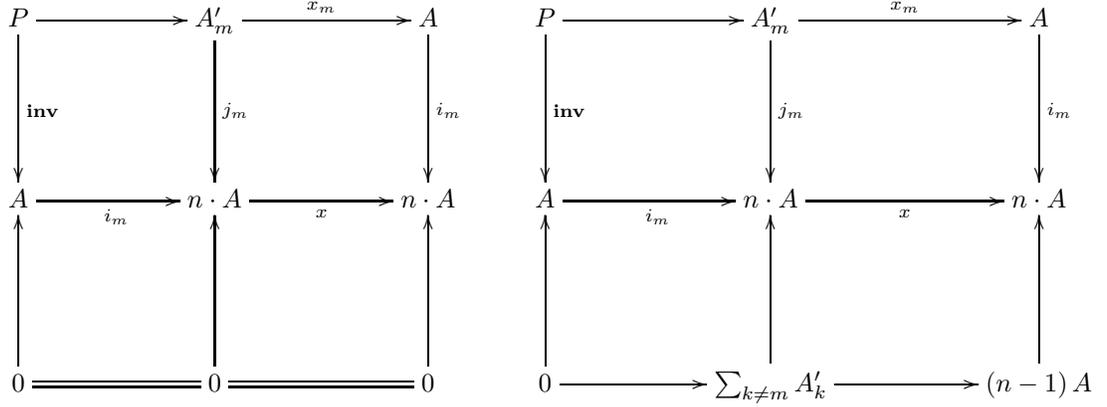
\[
\xymatrix{P\ar[rr]\ar[dd]^{\mathbf{inv}} &  & A'_{m}\ar[dd]^{j_{m}}\ar[rr]^{x_{m}} &  & A\ar[dd]^{i_{m}} & P\ar[rr]\ar[dd]^{\mathbf{inv}} &  & A'_{m}\ar[dd]^{j_{m}}\ar[rr]^{x_{m}} &  & A\ar[dd]^{i_{m}}\\
\\
A\ar[rr]_{i_{m}} &  & n\cdot A\ar[rr]_{x} &  & n\cdot A & A\ar[rr]_{i_{m}} &  & n\cdot A\ar[rr]_{x} &  & n\cdot A\\
\\
0\ar[uu]\ar@{=}[rr] &  & 0\ar[uu]\ar@{=}[rr] &  & 0\ar[uu] & 0\ar[rr]\ar[uu] &  & \sum_{k\neq m}A'_{k}\ar[uu]\ar[rr] &  & \left(n-1\right)A\ar[uu]
}
\]
where $A'_{m}$ is the indicated pullback (and the indicated arrow
is invertible since coproduct injections are pullback stable using
the right above). We first see that $A'_{m}=A$ (taking representatives
to avoid isomorphisms), since the indicated arrow above is invertible,
as well as that below
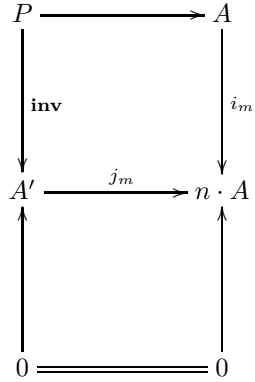
\[
\xymatrix{P\ar[dd]^{\mathbf{inv}}\ar[rr] &  & A\ar[dd]^{i_{m}}\\
\\
A'\ar[rr]^{j_{m}} &  & n\cdot A\\
\\
0\ar[uu]\ar@{=}[rr] &  & 0\ar[uu]
}
\]
also showing $i_{m}=j_{m}$ (as they are both coproduct injections
from the same object). Also noting codiagonals are pullback stable
we get
\[
\textnormal{id}=\textnormal{codiag}\cdot i_{m}=\textnormal{codiag}\cdot x\cdot i_{m}=\textnormal{codiag}\cdot i_{m}\cdot x_{m}=x_{m}
\]
as required, thus showing $x\colon nA\to nA$ is the identity.
\end{proof}
We now have an explicit description of the $n$-fibres pullback.
\begin{prop}
Let $\mathcal{E}$ be a locally cartesian closed category with disjoint
finite coproducts. The $n$-fibre pullback is then given by
\[
\xymatrix{n\cdot{}_{p_{n}}E\ar[rr]^{\textnormal{codiag}}\ar[d] &  & _{p_{n}}E\ar[dd]^{p_{n}}\\
n\cdot E^{\boxtimes_{B}n}\ar[d]_{n\cdot\delta}\\
E^{\boxtimes_{B}n}\times E\ar[rr]^{\pi}\ar[d] &  & E^{\boxtimes_{B}n}\ar[d]^{p}\\
E\ar[rr]_{p} &  & B
}
\]
\end{prop}

\begin{proof}
Suppose we are given a pullback as below
\[
\xymatrix{A\ar[rr]^{i_{m}} &  & {\displaystyle n\times A}\ar[rr]^{\pi_{2}}\ar[dd]_{g} &  & A\ar[dd]^{f}\\
\\
 &  & E\ar[rr]_{p} &  & B
}
\]
and pick an $1\leq m\leq n$, noting that $f=p\cdot g\cdot i_{m}$.
Write $g_{m}=g\cdot i_{m}$. Each $g_{m}$ is disjoint (a consequence
of $n\times A$ being a \emph{disjoint} sum of $n$ copies of $A$,
and the dense-closed factorization respecting sums). We may then factor
$f$ through $E^{\boxtimes_{B}n}$ as below. Thus we write this pullback
as
\[
\xymatrix{{\displaystyle n\times A}\ar[rr]^{\pi_{2}}\ar[dd]_{g_{1}+...+g_{n}} &  & A\ar[d]^{\left(g_{1},...,g_{n}\right)}\\
 &  & E^{\boxtimes_{B}n}\ar[d]^{p}\\
E\ar[rr]_{p} &  & B
}
\]
and we may then apply Lemma \ref{nfaclem} so recover the desired
factorization.
\end{proof}

\subsection{The homogeneous coreflections between polynomials}

Using these $n$-fibre pullbacks we may exhibit an infinite family
of adjunctions between bicategories of polynomials and those of order
$n$. These adjunctions correspond to taking the homogeneous terms.
Recall that the 2-cells here are cartesian maps of polynomials.
\begin{thm}
Let $\mathcal{E}$ be a locally cartesian closed category with disjoint
finite sums. Let $I,J\in\mathcal{E}$. Then the family of functors
\[
\mathbf{homterm}_{n}\colon\mathbf{Poly}\left(\mathcal{E}\right)\left(I,J\right)\to\mathbf{Poly}_{n}\left(\mathcal{E}\right)\left(I,J\right)
\]
 defined by
\[
\xymatrix{ &  &  &  &  &  & n\cdot{}_{p_{n}}E\ar[r]^{\textnormal{codiag}}\ar[d]\ar@{}[rd]|-{\textnormal{nfpb}} & _{p_{n}}E\ar[d]^{p_{n}}\\
 & E\ar[ld]_{s}\ar[r]^{p} & B\ar[rd]^{t} &  & \mapsto &  & E\ar[ld]_{s}\ar[r]_{p} & B\ar[rd]^{t}\\
I &  &  & J &  & I &  &  & J
}
\]
are coreflection right adjoints to the inclusion of order $n$ homogeneous
polynomials into general polynomials written $\mathbf{inc}_{n}\colon\mathbf{Poly}_{n}\left(\mathcal{E}\right)\left(I,J\right)\to\mathbf{Poly}\left(\mathcal{E}\right)\left(I,J\right)$
.
\end{thm}

\begin{proof}
We first note that the composite $\mathbf{homterm}_{n}\cdot\mathbf{inc}_{n}$
on $\mathbf{Poly}\left(\mathcal{E}\right)\left(I,J\right)$ is isomorphic
to the identity, as clearly the $n$-fibre pullback about $nB\to B$
is the square
\[
\xymatrix{nB\ar[d]_{\textnormal{id}}\ar[r] & B\ar[d]^{\textnormal{id}}\\
nB\ar[r] & B
}
\]
as this is trivially terminal. We specify the counit $\epsilon\colon\mathbf{inc}_{n}\cdot\mathbf{homterm}_{n}\to\textnormal{id}$
at a general polynomial $\left(s,p,t\right)$ to be the $n$-fibre
pullback
\[
\xymatrix{ & n\cdot{}_{p_{n}}E\ar[r]^{\textnormal{codiag}}\ar[dd]_{n\cdot p_{n}}\ar@{}[rdd]|-{\textnormal{nfpb}}\ar[ld] & _{p_{n}}E\ar[dd]^{p_{n}}\ar[rd]\\
I &  &  & J\\
 & E\ar[lu]_{s}\ar[r]_{p} & B\ar[ru]^{t}
}
\]
The triangle identities are left as an exercise for the reader.
\end{proof}
\begin{rem}
By doctrinal adjunction \cite{doctrinal} oplax structures on the
left adjoint biject with lax structures on the right adjoint. Provided
that there is a natural bicategory structure on $\mathbf{Poly}_{n}\left(\mathcal{E}\right)\left(I,J\right)$
such that the inclusion is oplax, it follows that each homogeneous
term functor
\[
\mathbf{homterm}_{n}\colon\mathbf{Poly}\left(\mathcal{E}\right)\left(I,J\right)\to\mathbf{Poly}_{n}\left(\mathcal{E}\right)\left(I,J\right)
\]
is a lax functor. This is the case when $n$ is zero or one. In the
case $n=1$, we recover the coreflection adjunction between $\mathbf{Poly}\left(\mathcal{E}\right)$
and $\mathbf{Span}\left(\mathcal{E}\right)$ first shown in the authors
Master's thesis \cite{mres}.
\end{rem}

\section{The $\left(I,J\right)=\left(1,1\right)$ derivatives of polynomials}

In this section we give a categorical version of McBride's definition
of derivatives of polynomials \cite{McBride}, and then explain why
a modified version of this definition involving the dense-mono localization,
allows all derivatives to exist.

\subsection{The $\left(I,J\right)=\left(1,1\right)$ products of polynomials}

Before one can define derivatives of polynomials, we we need the basic
description of products. There are many examples of products \cite{spivakpoly},
but we will only require the most basic of them.
\begin{defn}
Suppose we are given two polynomials $p$ and $p'$ as below (with
their associated polynomial functors) 
\[
\xymatrix{\mathbf{1} & E\ar[r]^{p}\ar[l] & B\ar[r] & \mathbf{1} & \mapsto & \left(\sum_{b\in B}X^{E_{b}}\right)\\
\mathbf{1} & E'\ar[r]^{p'}\ar[l] & B'\ar[r] & \mathbf{1} & \mapsto & \left(\sum_{b\in B'}X^{E'_{b}}\right)
}
\]
then the product of the two polynomials written $p\ast p'$ is

\[
\xymatrix{\mathbf{1} & EB'+E'B\ar[l]\myar{p\ast p'}{r} & BB'\ar[r] & \mathbf{1} & \mapsto & \left(\sum_{\left(b,b'\right)\in BB'}X^{E_{b_{1}}+E_{b_{2}}}\right)}
\]
\end{defn}

We will only need the following basic example of multiplication by
the linear polynomial $\mathbf{1}$.
\begin{example}
The product of 

\[
\xymatrix{\mathbf{1} & E\ar[r]^{p}\ar[l] & B\ar[r] & \mathbf{1} & \mapsto & \left(\sum_{b\in B}X^{E_{b}}\right)\\
\mathbf{1} & \mathbf{1}\ar[r]^{\textnormal{id}}\ar[l] & \mathbf{1}\ar[r] & \mathbf{1} & \mapsto & \left(X\right)
}
\]
is simply
\[
\xymatrix{\mathbf{1} & E+B\ar[l]\myar{p\ast\mathbf{1}}{r} & B\ar[r] & \mathbf{1} & \mapsto & \left(\sum_{b\in B}X^{E_{b}+1}\right)}
\]
\end{example}

McBride's definition of derivative (written down categorically) is
then as the linear exponential.
\begin{defn}
Let $\mathcal{E}$ be a locally cartesian closed category with disjoint
finite sums. The derivative functor $D\colon\mathbf{Poly}\mathcal{E}\left(1,1\right)\to\mathbf{Poly}\mathcal{E}\left(1,1\right)$
is then defined as the right adjoint to linear multiplication, that
is $\left(-\right)\ast\mathbf{1}\dashv D\left(-\right)$.
\end{defn}

We now explain categorically why (when a decidability condition is
satisfied) this derivative exists. It is worth noting that the candidate
derivative always exists, regardless of decidability. This formula
is also given in the case $\mathcal{E}=\mathbf{Set}$ in the notes
of Kock \cite{kocknotes}.
\begin{prop}
Suppose we are given a polynomial 
\[
\xymatrix{\mathbf{1} & E\ar[r]^{p}\ar[l] & B\ar[r] & \mathbf{1}}
\]
such that the diagonal map $\delta_{p}\colon E\to E\times_{B}E$ is
decidable as in Definition \ref{decidable}. Then the derivative $Dp$
exists, and is given by the polynomial
\[
\xymatrix@R=0.5em{ & E\boxtimes_{B}E\ar[r]^{\neg\delta_{p}}\ar[ld] & E\times_{B}E\myar{\pi_{1}}{r} & E\ar[dr]\\
\mathbf{1} &  &  &  & \mathbf{1}
}
\]
\end{prop}

\begin{proof}
To show this is the derivative, we must specify the unit $\eta_{p}\colon p\to D\left(p\times\mathbf{1}\right)$
and counit $\epsilon_{p}\colon Dp\times\mathbf{1}\to p$ components
of the adjunction$.$ To describe the counit $\epsilon$ at a component
$p$, we first note that $Dp\ast\mathbf{1}$
\[
\xymatrix{E\boxtimes_{B}E+E\ar[rr]^{Dp+\textnormal{id}} &  & E}
\]
is isomorphic to (by decidability of $p$)
\[
\xymatrix{E\times_{B}E\ar[rr]^{\pi_{1}} &  & E}
\]
so that the counit is simply the kernel pair
\[
\xymatrix{E\times_{B}E\ar[rr]^{p}\ar[d]_{p} &  & E\ar[d]^{p}\\
E\ar[rr]_{p} &  & B
}
\]
Now the unit is given by $\eta_{p}\colon p\to D\left(p\ast\mathbf{1}\right)$
\[
\xymatrix{E\ar@/^{1pc}/[rrrr]^{p}\ar@{..>}[rr]\ar@{..>}[dd] &  & B\times_{B}\left(E+B\right)\ar[rr]_{\pi_{1}}\ar[dd]^{\iota_{B\times_{B}\left(E+B\right)}} &  & B\ar[dd]^{\iota_{B}}\\
\\
\left(E+B\right)\boxtimes_{B}\left(E+B\right)\ar[rr]_{\neg\textnormal{diag}} &  & \left(E+B\right)\times_{B}\left(E+B\right)\myard{\pi_{1}}{rr} &  & E+B
}
\]
where the dotted arrows are constructed by the following argument.
We use the pseudocommutativity of the square
\[
\xymatrix{\mathcal{E}/\left(E_{1}+E_{2}\right)\ar[rr]^{\Pi_{p_{1}+p_{2}}}\ar[d]_{\simeq} &  & \mathcal{E}/\left(B_{1}+B_{2}\right)\\
\mathcal{E}/E_{1}\times\mathcal{E}/E_{2}\ar[rr]_{\Pi_{p_{1}}\times\Pi_{p_{2}}} &  & \mathcal{E}/B_{1}\times\mathcal{E}/B_{2}\ar[u]_{\simeq}
}
\]
to justify taking the sum of the distributivity pullbacks
\[
\xymatrix{\mathbf{0}\ar[rr]^{!}\ar@{=}[d] &  & E\boxtimes_{B}E\ar[dd]^{\neg\textnormal{diag}} & \mathbf{0}\ar[rr]^{!}\ar@{=}[d] &  & \overbrace{B\boxtimes_{B}B}^{0}\ar[dd]^{\neg\textnormal{diag}}\\
\mathbf{0}\ar[d]_{!}\ar@{}[rr]|-{\textnormal{dpb}} &  & \; & \mathbf{0}\ar[d]_{!}\ar@{}[rr]|-{\textnormal{dpb}} &  & \;\\
E\ar[rr]_{\textnormal{diag}} &  & E\times_{B}E & B\ar[rr]_{\textnormal{diag}} &  & B\times_{B}B
}
\]
and
\[
\xymatrix{\mathbf{0}\ar[rr]^{!}\ar@{=}[d] &  & \overbrace{E\times_{B}B}^{E}\ar@{=}[dd] & \mathbf{0}\ar[rr]^{!}\ar@{=}[d] &  & \overbrace{B\times_{B}E}^{E}\ar@{=}[dd]\\
\mathbf{0}\ar[d]_{!}\ar@{}[rr]|-{\textnormal{dpb}} &  & \; & \mathbf{0}\ar[d]_{!}\ar@{}[rr]|-{\textnormal{dpb}} &  & \;\\
0\ar[rr]_{!} &  & E\times_{B}B & 0\ar[rr]_{!} &  & B\times_{B}E
}
\]
giving since $B\times_{B}E$ is the only non-disjoint component common
of the sum $\left(E+B\right)\times_{B}\left(E+B\right)$ the diagram
\[
\xymatrix{\mathbf{0}\ar[rr]^{!}\ar@{=}[d] &  & \overbrace{\left(E+B\right)\boxtimes_{B}\left(E+B\right)}^{E\boxtimes_{B}E+E\times_{B}B+B\times_{B}E}\ar[dd]^{\neg\textnormal{diag}} &  & \overbrace{B\times_{B}E}^{E}\ar@{..>}[ll]\ar@{..>}[dd]\\
\mathbf{0}\ar[d]_{!}\ar@{}[rr]|-{\sum\textnormal{dpb}} &  & \;\\
E+B\myar{\textnormal{diag}}{rr} &  & \left(E+B\right)\times_{B}\left(E+B\right) &  & \underbrace{B\times_{B}\left(E+B\right)}_{B\times_{B}E+B\times_{B}B}\myar{\iota_{B\times_{B}\left(E+B\right)}}{ll}
}
\]
The triangle identities are left as an exercise.
\end{proof}
\begin{rem}
We may work out the polynomial functor for the derivative via the
calculation
\[
\begin{aligned}\mathbf{Ext}\left(D\left(p\right)\right) & =\sum_{e\in E}X^{E_{pe}-1}\\
 & =\sum_{b\in E}\sum_{e\in E_{b}}X^{E_{b}-1}\\
 & =\sum_{b\in E}E_{b}X^{E_{b}-1}\\
 & =D\left(\mathbf{Ext}\left(p\right)\right)
\end{aligned}
\]
where $\text{\ensuremath{\mathbf{Ext}}}$ is the extension to polynomial
functors, and $D$ is the derivative.
\end{rem}

Now using the localization from earlier to account decidability we
get the following main result.
\begin{thm}
Let $\mathcal{E}$ be a locally cartesian closed category with disjoint
finite sums. Let $\mathcal{E}\left[\mathcal{W}^{-1}\right]$ denote
the dense mono localization of $\mathcal{E}$. Then the pseudofunctor
\[
\mathbf{Poly}\mathcal{E}\to\mathbf{Poly}\mathcal{E}\left[\mathcal{W}^{-1}\right]
\]
preserves $\left(I,J\right)=\left(1,1\right)$ derivatives, and all
such derivatives exist in $\mathbf{Poly}\mathcal{E}\left[\mathcal{W}^{-1}\right]$.
\end{thm}

\section{The $\left(I,J\right)=\left(1,1\right)$ order-$F$ candidate derivatives
of polynomials}

In the previous section we considered the derivative of order $\mathbf{1}\in\mathcal{E}$,
but more generally one may consider the derivative of a general order
$F\in\mathcal{E}$.
\begin{defn}
Let $\mathcal{E}$ be a locally cartesian closed category with disjoint
finite sums. The order $F$ derivative functor $D^{F}\colon\mathbf{Poly}\mathcal{E}\left(1,1\right)\to\mathbf{Poly}\mathcal{E}\left(1,1\right)$
is then defined as the right adjoint to $F$ multiplication, that
is $\left(-\right)\ast F\dashv D\left(-\right)$.
\end{defn}

\begin{rem}
Here we are using that an object $F$ is viewed as the polynomial
\[
\xymatrix{\mathbf{1} & F\ar[r]\ar[l] & \mathbf{1}\ar[r] & \mathbf{1}}
\]
and thus $p\ast F$ is 
\[
\xymatrix{\mathbf{1} & E+F\cdot B\ar[r]\ar[l] & B\ar[r] & \mathbf{1}}
\]
\end{rem}

Whilst it is unclear if this order $F$ derivative will always exist,
what we can say is that the \emph{candidate derivative} does always
exist.
\begin{defn}
Let $\mathcal{E}$ be a locally cartesian closed category with disjoint
finite sums. The order $F$ candidate derivative of a polynomial
\[
\xymatrix{\mathbf{1} & E\ar[r]^{p}\ar[l] & B\ar[r] & \mathbf{1}}
\]
is defined as 
\[
\xymatrix@R=0.5em{ & E\boxtimes_{B}E^{\boxtimes_{B}F}\ar[ld]\myar{\neg\left(\textnormal{ev},\textnormal{id}\right)}{rr} &  & E\times_{B}E^{\boxtimes_{B}F}\myar{\pi_{2}}{rr} &  & E^{\boxtimes_{B}F}\ar[dr]\\
\mathbf{1} &  &  &  &  &  & \mathbf{1}
}
\]
where $E^{\boxtimes_{B}F}$ is defined by the pair of distributivity
pullbacks
\[
\xymatrix{FE\ar@/_{1.2pc}/[rdd]_{\textnormal{id}}\ar[rrr]\ar@{..>}[rd] &  &  & E\ar[lddd]\ar@{..>}[ld]\\
 & FE^{\times_{B}F}\ar[d]_{F\cdot\textnormal{ev}}\ar[r] & E^{\times_{B}F}\ar[dd] &  &  & \mathbf{0}\ar[r]^{!}\ar@{=}[d] & E^{\boxtimes_{B}F}\ar[dd]\\
 & FE\ar[d]\ar@{}[r]|-{\textnormal{dpb}}\ar[d]_{F\cdot p} & \; &  &  & \mathbf{0}\ar[d]_{!}\ar@{}[r]|-{\textnormal{dpb}} & \;\\
 & FB\ar[r] & B &  &  & FE\ar[r] & E^{\times_{B}F}
}
\]
and $\neg\left(\textnormal{ev},\textnormal{id}\right)$ is defined
by the distributivity pullback
\[
\xymatrix{\mathbf{0}\ar[rrrr]^{!}\ar@{=}[d] &  &  &  & E\boxtimes_{B}E^{\boxtimes_{B}F}\ar[dd]^{\neg\left(\textnormal{ev},\textnormal{id}\right)}\\
\mathbf{0}\ar[d]_{!} &  &  &  & \;\\
F\times E^{\boxtimes_{B}F}\ar[rrrr]_{\left(\textnormal{ev},\textnormal{id}\right)} &  &  &  & E\times_{B}E^{\boxtimes_{B}F}
}
\]
using the commutativity of

\[
\xymatrix{ & F\times E^{\times_{B}F}\ar[r]^{\textnormal{ev}}\ar@{..>}[rd] & E\ar[rd]^{p}\\
F\times E^{\boxtimes_{B}F}\ar[ru]\ar[r] & E^{\boxtimes_{B}F}\ar[r] & E^{\times_{B}F}\ar[r] & B
}
\]
\end{defn}

\begin{rem}
We would expect $\left(\textnormal{ev},\textnormal{id}\right)$ to
decidable to have a derivative, but this may not be enough. A possible
extra condition is disjointness on exponentials meaning that we have
pullbacks as on the left below
\[
\xymatrix{ &  &  & FB\ar[rrrr]\ar@/_{1pc}/[rddd]\ar[rd] &  &  &  & B\ar@/^{1pc}/[dddl]\ar[dl]\\
0\ar[rr]\ar[d]\ar@{}[rdr]|-{\textnormal{pb}} &  & B\ar[d]^{\overline{\textnormal{id}_{FB}}} &  & F\left(FB\right)^{F}\ar[rr]\ar[d]\ar@{}[rrdd]|-{\textnormal{dpb}} &  & \left(FB\right)^{F}\ar[dd]\\
FB\ar[rr]_{\textnormal{diag}} &  & \left(FB\right)^{F} &  & F\left(FB\right)\ar[d]\\
 &  &  &  & F\ar[rr] &  & 1
}
\]
where $\left(FB\right)^{F}$ is defined as on the right above. 
\end{rem}

\section{Acknowledgments}

The author thanks the members of Masaryk University Algebra Seminar,
Leeds University Seminar, and Coimbra University Algebra Seminar for
their questions and comments. 

This work was supported by the Operational Programme Research, Development
and Education Project \textquotedblleft Postdoc@MUNI\textquotedblright{}
(No. CZ.02.2.69/0.0/0.0/18\_053/0016952).

\section{Future work}

It remains to investigate if the partial derivatives and total derivatives
can be simply expressed in a similar way to the above. Moreover, if
differentiation (seen as a linear operator) can be viewed as a span,
and if derivatives can be defined more simply as an operation with
this span. It also remains to understand if one has a two dimensional
cartesian differential category structure on polynomials, though this
would likely require generalizing the the theory of cartesian differential
categories \cite{cartesiandifferential} to two dimensions. We may
also consider possible versions of the adjoints to derivatives which
appear in the case of species \cite{adjointderivative}.

Finally, there are possible connections with results in the $\infty$-categorical
literature on left exact localizations of topoi \cite{leftexactloc}
and modalities \cite{modal} to be explored.

\bibliographystyle{plain}
\bibliography{references}

\begin{thebibliography}{10}

\bibitem{McBride}
Michael Abbott, Thorsten Altenkirch, Neil Ghani, and Conor McBride.
\newblock Derivatives of containers.
\newblock In {\em Typed lambda calculi and applications ({V}alencia, 2003)},
  volume 2701 of {\em Lecture Notes in Comput. Sci.}, pages 16--30. Springer,
  Berlin, 2003.

\bibitem{leftexactloc}
Mathieu Anel, Georg Biedermann, Eric Finster, and Andr\'e Joyal.
\newblock Left-exact localizations of $\infty$-topoi ii: Grothendieck
  topologies, 2022.

\bibitem{cartesiandifferential}
R.~F. Blute, J.~R.~B. Cockett, and R.~A.~G. Seely.
\newblock Cartesian differential categories.
\newblock {\em Theory Appl. Categ.}, 22:622--672, 2009.

\bibitem{extensive}
Aurelio Carboni, Stephen Lack, and R.~F.~C. Walters.
\newblock Introduction to extensive and distributive categories.
\newblock {\em J. Pure Appl. Algebra}, 84(2):145--158, 1993.

\bibitem{gambinokock}
Nicola Gambino and Joachim Kock.
\newblock Polynomial functors and polynomial monads.
\newblock {\em Math. Proc. Cambridge Philos. Soc.}, 154(1):153--192, 2013.

\bibitem{doctrinal}
G.~M. Kelly.
\newblock Doctrinal adjunction.
\newblock In {\em Category {S}eminar ({P}roc. {S}em., {S}ydney, 1972/1973)},
  pages 257--280. Lecture Notes in Math., Vol. 420. Springer, Berlin, 1974.

\bibitem{kocknotes}
Joachim Kock.
\newblock Notes on polynomial functors.
\newblock (unpublished), 2009.

\bibitem{martinlof}
Per Martin-L\"of.
\newblock Intuitionistic type theory.
\newblock {\em Studies in Proof Theory}, 1984.

\bibitem{adjointderivative}
Dayanand~S. Rajan.
\newblock The adjoints to the derivative functor on species.
\newblock {\em J. Combin. Theory Ser. A}, 62(1):93--106, 1993.

\bibitem{modal}
Egbert Rijke, Michael Shulman, and Bas Spitters.
\newblock {Modalities in homotopy type theory}.
\newblock {\em arXiv eprint}, 2020.
\newblock Available at \url{https://arxiv.org/abs/1706.07526}, Version 6.

\bibitem{spivakpoly}
David Spivak.
\newblock {A reference for categorical structures on Poly}.
\newblock {\em arXiv eprint}, 2022.
\newblock Available at \url{https://arxiv.org/abs/2202.00534}, Version 2.

\bibitem{glehn}
Tamara von Glehn.
\newblock {\em Polynomials and models of type theory}.
\newblock PhD thesis, University of Cambridge, 2015.

\bibitem{mres}
Charles Walker.
\newblock Local reflections between relations, spans and polynomials.
\newblock $\text{MRes thesis}$, Macquarie University, 2015.

\bibitem{weber}
Mark Weber.
\newblock Polynomials in categories with pullbacks.
\newblock {\em Theory Appl. Categ.}, 30:No. 16, 533--598, 2015.

\end{thebibliography}

\end{document}